\documentclass[a4paper,10pt,reqno]{amsart}
\usepackage{amsmath,amssymb,amsfonts,graphicx,epsfig,color,url}
\usepackage{latexsym,mathrsfs,cite}

\numberwithin{equation}{section}

\newtheorem{theorem}{Theorem}[section]
\newtheorem{proposition}[theorem]{Proposition}
\newtheorem{lemma}[theorem]{Lemma}

\theoremstyle{remark}

\usepackage{float}
\usepackage{mathtools}
\usepackage{hyperref} 
\usepackage{placeins}
\usepackage{comment}
\usepackage[title,titletoc,toc]{appendix}

\usepackage{graphicx}
\usepackage{subcaption}
\usepackage[labelformat=simple]{subcaption}

\allowdisplaybreaks

\title[Spectral analysis for QHD with nonlinear viscosity]{Spectral analysis of dispersive shocks for quantum hydrodynamics with nonlinear viscosity}
\author{Corrado Lattanzio and Delyan Zhelyazov}

\address[Corrado Lattanzio]{DISIM, Department of Information Engineering, Computer Science and Mathematics \\ University of L'Aquila, Italy}
\email{corrado@univaq.it}
\address[Delyan Zhelyazov]{DISIM, Department of Information Engineering, Computer Science and Mathematics \\ University of L'Aquila, Italy}\email{delyanatanasov.zhelyazov@univaq.it}

\begin{document}

\begin{abstract}
In this paper we investigate spectral stability of traveling wave solutions to 1-$D$ quantum hydrodynamics system with nonlinear viscosity in the   $(\rho,u)$, that is, density and velocity, variables. We derive a sufficient condition for the stability of the essential spectrum and we estimate the maximum modulus of eigenvalues with non-negative real part. In addition, we present numerical computations of the Evans function in sufficiently large domain of the unstable half-plane and show numerically that its winding number is (approximately) zero, thus giving a numerical evidence of point spectrum stability.
\end{abstract}

\keywords{quantum hydrodynamics, traveling waves, spectral stability, dispersive-diffusive shock waves}
\subjclass[2010]{76Y05, 35Q35}

\maketitle

\section{Introduction}
\label{Introduction}
The aim of this paper is to investigate stability properties of the following quantum hydrodynamics (QHD) system with nonlinear viscosity:
\begin{equation}
\label{eq_sys_n}
\left\{
\begin{array}{ll}
\rho_t+m_x=0,\\
m_t+\Big{(}\frac{m^2}{\rho}+p(\rho) \Big{)}_x=\epsilon \mu \rho \Big(\frac{m_{x}}{\rho}\Big)_x+\epsilon^2 k^2 \rho \Big(\frac{(\sqrt{\rho})_{xx}}{\sqrt{\rho}}\Big{)}_x.
\end{array}
\right.
\end{equation}
Here $\rho \geq 0$ is the density, $m = \rho u$ is the momentum, where $u$ denotes the fluid velocity, and $0 < \epsilon \ll 1$ and $\mu$, $k>0$ are constants. Moreover, $\epsilon \mu$ and $\epsilon^2 k^2 $ are the viscosity and dispersion coefficients, respectively, and $p(\rho) = \rho^\gamma$ with $\gamma \geq 1$ is the pressure. The form of the dispersive term is known as the Bohm potential, while the nonlinear viscosity chosen here appears in the theory of superfluidity; see, for instance,  \cite{Khalatnikov}, page 109. This term  describes the interactions between  a super fluid and a normal fluid; in addition, it can also be interpreted as describing the interactions of the fluid with a background.
The first studies of models with dispersive terms are \cite{Sagdeev, Gurevich}; see also \cite{Gurevich1, Nov, Hoefer}. Moreover, quantum hydrodynamic systems have been considered from a mathematical perspective in \cite{AM1,AM2,AMtf,AMDCDS,AS,Michele1,Michele,DM,DM1,DFM,GLT,BGL-V19}.

Specifically, in what follows we shall deal with traveling wave solutions, or \emph{dispersive shocks}, for the system \eqref{eq_sys_n}, namely solutions depending on the ratio $(x-st)/\epsilon$, where the constant $s$ stands for the speed of the traveling wave, with given end states at $\pm\infty$. The existence of such solutions, under appropriate conditions on the end states, is investigated in \cite{Zhelyazov1}. 
More precisely, for strictly positive end states for the density, the corresponding profile stays away from vacuum (in $\rho$). Therefore, the velocity $u$ is well defined and the system can be recast in the $(\rho,u)$ variables. This reformulation is also justified by the fact that the QHD system is related to the Schr\"odinger equation, and the velocity can be written in terms of the phase of the associated wavefunction as $u = \phi_x$; see for instance \cite{Gasser}.
Hence, we divide the second equation of \eqref{eq_sys_n} by $\rho$ to obtain
\begin{align}
&\rho_t+(\rho u)_x=0, \label{eq_sys1_1}\\*
&\frac{(\rho u)_t}{\rho}+\frac{1}{\rho}\Big(\rho u^2 + p(\rho) \Big)_x=\epsilon \mu \Big(\frac{(\rho u)_{x}}{\rho}\Big)_x+\epsilon^2 k^2 \Big(\frac{(\sqrt{\rho})_{xx}}{\sqrt{\rho}}\Big{)}_x. \label{eq_sys1_2}
\end{align}
Let us define the enthalpy $h(\rho)$ by 
 \begin{equation}\label{eq:ent}
 h(\rho) = \begin{dcases}
 \ln \rho, & \gamma = 1\\
 \frac{\gamma}{\gamma-1}\rho^{\gamma-1}, & \gamma>1;
 \end{dcases}
 \end{equation}
see, for instance,  \cite{Gasser}. Then $h$   satisfies the identity
\begin{equation}
h(\rho)_x=\frac{1}{\rho}(p(\rho))_x\nonumber
\end{equation} 
and  the momentum equation can be simplified by using the continuity equation as follows:
\begin{align*}
&\frac{(\rho u)_t}{\rho}+\frac{1}{\rho}(\rho u^2)_x=\frac{\rho_t u}{\rho}+u_t + \frac{1}{\rho}\rho_x u^2 + (u^2)_x\nonumber\\
&=-\frac{(\rho u)_x u}{\rho}+u_t + \frac{1}{\rho}\rho_x u^2 + (u^2)_x=u_t + \frac{(u^2)_x}{2}.
\end{align*}
As a consequence, the system \eqref{eq_sys1_1}-\eqref{eq_sys1_2} can be rewritten in conservative form using the velocity and the enthalpy as follows:
\begin{align}
&\rho_t+(\rho u)_x=0,\label{eq_sys2_1}\\
&u_t + \frac{(u^2)_x}{2}+(h(\rho))_x=\epsilon \mu \Big(\frac{(\rho u)_{x}}{\rho}\Big)_x+\epsilon^2 k^2 \Big(\frac{(\sqrt{\rho})_{xx}}{\sqrt{\rho}}\Big{)}_x. \label{eq_sys2_2}
\end{align}
In the present work, we shall study the spectrum of the linearization of \eqref{eq_sys2_1}-\eqref{eq_sys2_2} around traveling wave profiles
\begin{equation}\label{eq:profintr}
\rho(t,x)=P\Big(\frac{x-s t}{\epsilon}\Big)^2,\mbox{ }u(t,x)=U\Big(\frac{x-st}{\epsilon}\Big).
\end{equation}
As a final remark, it is worth observing that the present spectral analysis applies also to non-monotone shocks.

Stability analysis of traveling wave solutions of partial differential equations is a widely  studied problem. In particular, for the case of this kind of hydrodynamic models involving dispersion terms, we recall here \cite{Humpherys}, where the spectral stability of traveling wave profiles for the $p-$system with real viscosity and linear capillarity has been discussed. Moreover, spectral analysis of the linearization around dispersive shocks for a variant of the QHD system \eqref{eq_sys_n} with linear viscosity can be found in \cite{Zhelyazov}, and the related Evans function computations in \cite{LMZ2020}.

The remaining part of this paper is organized as follows.
In Section \ref{linearization}  we show that the essential spectrum of the linearized operator around a profile is stable for subsonic or sonic end states. In Section \ref{point_spectrum} we estimate the maximum modulus of possible eigenvalues with non-negative real part, giving an explicit bound for the constant, and using this bound we perform numerics about the Evans function, providing numerical evidence for point spectrum stability of a non-monotone profile.

\section{Linearization and essential spectrum}\label{linearization}
We start by performing a linearization of system \eqref{eq_sys2_1}-\eqref{eq_sys2_2} around a profile \eqref{eq:profintr} with end states 
\begin{equation*}
\rho^{\pm}=\lim_{y\rightarrow \pm \infty}P(y)^2,\mbox{ }u^{\pm}=\lim_{y\rightarrow \pm \infty}U(y),
\end{equation*}
 For the sake of completeness, we state here  the existence theory for such profiles established in  \cite{Zhelyazov1}. To this end, let us recall that the Rankine--Hugoniot  conditions for a shock $(\rho^{\pm}, u^\pm;s)$ of the underlying  system 
 \begin{align}
&\rho_t+(\rho u)_x=0, \label{euler1} \\
&u_t + \frac{(u^2)_x}{2}+(h(\rho))_x=0,
\label{euler2}
\end{align}
 read
\begin{equation}\label{eq:RH}
\begin{aligned}
s((P^2)^+-(P^2)^-) =(P^2 u)^+ - (P^2 u)^-,\\
s(u^+-u^-)=\Big(\frac{u^2}{2}+h(P^2)\Big)^+-\Big(\frac{u^2}{2}+h(P^2)\Big)^-.
\end{aligned}
\end{equation}
Moreover, the characteristic speeds $\lambda_{1,2}(\rho,u)$ of the  hyperbolic system \eqref{euler1}-\eqref{euler2} are given by
\begin{equation*}
\lambda_1(W) = u - c_s(\rho),\mbox{ }\lambda_2(W) = u + c_s(\rho),
\end{equation*}
where we used the notation $c_s(\rho) = \sqrt{\rho h'(\rho)}\geq0$ for the sound speed.  Indeed, 
from the definition of the   enthalpy in \eqref{eq:ent} we readily obtain
 \begin{equation*}
 h'(\rho) = \begin{dcases}
 \frac{1}{\rho}, & \gamma = 1\\
 {\gamma}\rho^{\gamma-2}, & \gamma>1.
 \end{dcases}
 \end{equation*}
Therefore, $h'(\rho)\geq 0$, and in fact $h'(\rho)> 0$ for any $\rho>0$, 
 and the sound speed $c_s(\rho)$ is well defined and non-negative for any $\rho$, and  strictly positive for $\rho>0$.
Then, we recall that a discontinuity $(\rho^{\pm}, u^\pm;s)$  verifyng  the Rankine-Hugoniot conditions \eqref{eq:RH} is a Lax $k$--shock, $k=1,2$, if
\begin{equation*}
\lambda_k(\rho^{+}, u^+)<s<\lambda_k(\rho^{-}, u^- ).
\end{equation*}
Moreover, the state $(\rho^\pm,u^\pm)$ is referred to as supersonic (resp.\ subsonic; sonic) if $|u^\pm| > c_s(\rho^\pm)$ (resp.\ $|u^\pm| < c_s(\rho^\pm)$; $|u^\pm| = c_s(\rho^\pm)$). We are now ready to state the main existence result for profiles to \eqref{eq_sys2_1}-\eqref{eq_sys2_2} proved in \cite{Zhelyazov1}. 
\begin{theorem}\label{theo:exprof}
Suppose the end states $(\rho^{\pm}$, $u^{\pm})$ and the speed $s$ satisfy the Rankine--Hugoniot conditions \eqref{eq:RH} with $\rho^{\pm}>0$, and $(\rho^{\pm}$, $u^{\pm}; s)$ defines
\begin{enumerate}
\item[\textit{(i)}]
 a Lax 2--shock with a subsonic right state;
\item[\textit{(ii)}]
  a Lax 1--shock with a subsonic left state. 
\end{enumerate}
Then there exists a traveling wave profile connecting 
$(\rho^{-}$, $u^{-})$ to $(\rho^{+}$, $u^{+})$.
\end{theorem}
It is worth to observe that the resulting profile may be non-monotone in $\rho$, depending on the magnitude of the ratio $\mu/k$, yet it stays away from vacuum; see \cite{Zhelyazov1} for details.

Changing the variables $\tau=t/\epsilon$, $y=(x-st)/\epsilon$, and denoting by $R(y)=P(y)^2$ and by $(\rho,u)$ the deviation from $(R,U)$, we obtain the following  full linearized  operator around the profile  
\begin{equation}
\label{operator_Lnd}L
\begin{bmatrix}
\rho\\
u
\end{bmatrix}
\\=\begin{bmatrix}
s \rho_y  -(R u + U \rho)_y\\
s u_y - (U u)_y - (\frac{dh}{dR}(R) \rho)_y  \\
 + \hfill \mu\big{(}(R^{-1}(R u + U \rho)_y)_y-(R^{-2} (R U)_y \rho)_y\big{)} 
  +k^2 L_Q \rho
\end{bmatrix},
\end{equation}
where
\begin{equation*}
L_{Q} \rho=\frac{1}{2}(R^{-1/2}(R^{-1/2} \rho)_{yy})_y-\frac{1}{2}(R^{-3/2}(R^{1/2})_{yy} \rho)_y,
\end{equation*}
and associated eigenvalue problem  
\begin{equation}
\label{eq_variable_coeff}
\lambda \begin{bmatrix}
\rho\\
u
\end{bmatrix} = L\begin{bmatrix}
\rho\\
u
\end{bmatrix}.
\end{equation}
With the notation
\begin{equation*}
R^{\pm}=\lim_{y\rightarrow \pm \infty}R(y),\mbox{ }U^{\pm}=\lim_{y\rightarrow \pm \infty}U(y)
\end{equation*}
for the end states, 
the asymptotic operators  at $\pm\infty$ for \eqref{operator_Lnd}  are given by
\begin{equation*}
L_{\pm\infty}
\begin{bmatrix}
\rho\\
u
\end{bmatrix}
=\begin{bmatrix}
(s-U^{\pm}) \rho'-R^{\pm} u'\\
(s-U^{\pm})u'-\frac{dh}{dR}(R^{\pm})\rho'+\mu\Big{(}u''+\frac{U^{\pm}}{R^{\pm}}\rho''\Big{)}+\frac{k^2}{2}\frac{\rho'''}{R^{\pm}}
\end{bmatrix},
\end{equation*}
where $'$ denotes $d/dy$. We  rewrite the eigenvalue problem  associated to the asymptotic operators 
\begin{equation*}
\lambda \begin{bmatrix}
\rho\\
u
\end{bmatrix} = L_{\pm\infty}\begin{bmatrix}
\rho\\
u
\end{bmatrix}
\end{equation*}
as the following  first order system
\begin{equation}
V'=M^{\pm}V.
\label{system_at_infinity}
\end{equation}
In \eqref{system_at_infinity}, 
the limit matrices are given by
\begin{equation}
\label{mat_Mn}
M^{\pm} = \begin{bmatrix}
0 & 0 & 1 & 0\\
-\frac{\lambda}{R^\pm} & 0 & \frac{s-U^\pm}{R^\pm} & 0\\
0 & 0 & 0 & 1\\
\frac{2(s-U^\pm)\lambda}{k^2} & \frac{2 R^\pm \lambda}{k^2} & \frac{2}{k^2}(R^\pm \frac{dh}{dR}(R^\pm) - (s-U^\pm)^2 + \mu \lambda) & -\frac{2 s \mu}{k^2}
\end{bmatrix}
\end{equation}
and $V=[\rho,u,u_1,u_2]^T$ with $\rho' = u_1$ and $u_1' = u_2$.
\subsection{Essential spectrum and consistent splitting}
\label{section_essential_spectrum}
The spectrum of the linearized operator $L$ consists of the essential spectrum and the point spectrum; in this section, we shall investigate the stability of the former.  In particular, we shall obtain sufficient conditions on the end states so that  essential spectrum is confined on the (stable) left half-plane $\Re\lambda\leq0$. To this end, 
let us consider the characteristic equation $\det(\nu Id-M^\pm)=0$ of $M^\pm$, that is
\begin{equation}
\label{char_equation}
\nu^4+\frac{2 s \mu}{k^2}\nu^3+\frac{2}{k^2}((s-U^\pm)^2- R^\pm h'(R^\pm)-\lambda \mu)\nu^2+\frac{4(U^\pm-s)}{k^2}\lambda \nu+\frac{2 \lambda^2}{k^2}=0.
\end{equation}
Setting $\nu = i \xi$, $\xi \in \mathbb{R}$, in \eqref{char_equation} and dividing by $2/k^2$, we obtain the   dispersion relation:
\begin{equation}
\label{disper_rel2}
\lambda^2 + (\mu \xi^2 - 2 i \xi (s-U^\pm))\lambda + (R^\pm h'(R^\pm)-(s-U^\pm)^2)\xi^2+\frac{k^2}{2}\xi^4 - i s \mu \xi^3=0.
\end{equation}
\begin{proposition}\label{prop:1}
If the end states are subsonic or sonic, then the curves $\lambda(\xi)$ solving \eqref{disper_rel2} are in the closed left half-plane. Moreover, if $\xi \neq 0$, then $\Re{\lambda_{1,2}}<0$.
\end{proposition}
\begin{proof}
To simplify notation in this proof we are going to drop the superscript of $R$ and $U$.\\
The roots of the dispersion relation \eqref{disper_rel2} are (see Figure \ref{fig_essential_spectrum2})
\begin{equation*}
\lambda_{1,2}=\frac{-\mu \xi^2 + 2 i \xi(s-U) \pm \sqrt{D}}{2},
\end{equation*}
where the discriminant is
\begin{align*}
D&=(\mu \xi^2 - 2 i \xi (s-U))^2 - 4\Big{(}(R h'(R)-(s-U)^2)\xi^2+\frac{k^2}{2}\xi^4 - i s \mu \xi^3\Big{)}\\
&=p+iq,
\end{align*}
with
\begin{align*}
p&=- 4 R h'(R) \xi^2+(\mu^2 - 2 k^2)\xi^4,\\
q&=4 \mu U  \xi^3.
\end{align*}
\begin{figure}[H]
\begin{center}
\includegraphics[scale=0.8]{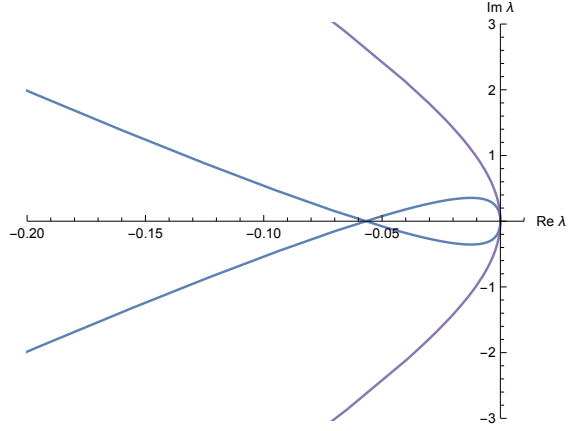}
\end{center}
\caption{The bound for the essential spectrum for parameters $k = \sqrt{2}$, $\mu = 0.1$, $s = 1$, $\gamma = 3/2$, $R = 0.5$, $U = -0.746$.}
\label{fig_essential_spectrum2}
\end{figure}
Clearly, if $\xi=0$, then the roots of \eqref{disper_rel2} are $\lambda_{1,2}=0$.

Suppose now $\xi \neq 0$. The condition 
\begin{equation}
\label{cond_stab1n}
 -\mu \xi^2 + |\Re \sqrt{D}| < 0 
 \end{equation}
 guarantees that $\lambda_{1,2}$ are in the left half-plane. Since $\mu \xi^2 > 0$, it is equivalent to
 \begin{equation}
 \label{cond_stab2n}
 (\Re \sqrt{D})^2 < \mu^2 \xi^4.
 \end{equation}
 By direct inspection we obtain
 \begin{equation*}
 \Re \sqrt{D} = \frac{\sqrt{2}}{2} \sqrt{\sqrt{p^2+q^2}+p}
 \end{equation*}
 and the condition \eqref{cond_stab2n} is equivalent to
 \begin{equation}
 \label{cond_stab3n}
 \sqrt{p^2+q^2} < 2 \mu^2 \xi^4 - p.
 \end{equation}
Since $h'(R)\geq 0$ implies in particular
 \begin{equation*}
 2 \mu^2 \xi^4 - p = 4 R h'(R) \xi^2 + (\mu^2+2 k^2) \xi^4 >0,
 \end{equation*}
  \eqref{cond_stab3n} is equivalent to
 \begin{equation*}
 p^2+q^2<(2 \mu^2 \xi^4 - p)^2,
 \end{equation*}
 that is
 \begin{equation*}
 B : = q^2 + 4 \mu^2 p \xi^4 -4 \mu^4 \xi^8 <0.
 \end{equation*}
Since the end states are subsonic or sonic, after squaring the corresponding inequality $|U| \leq c_s(R)$, we end up to 
 $U^2 - R h'(R) \leq 0$. Finally,
  \begin{equation*}
 \frac{B}{8 \mu^2} = 2 (U^2 - R h'(R)) \xi^6 - k^2 \xi^8<0
 \end{equation*}
and  \eqref{cond_stab1n} holds, concluding the proof.
\end{proof}

In the next proposition we examine the behavior of roots to \eqref{char_equation} to conclude in particular  \emph{consistent splitting}:  on the right of the curves $\lambda(\xi)$ solving \eqref{disper_rel2},
 i.e.\ the  values of $\lambda$ such that real part of roots of \eqref{char_equation} is zero, we have two roots of \eqref{char_equation} with positive real part and two roots with negative real part.
\begin{proposition}\label{prop:2}
If $\mu^2 \neq 2 k^2$,  on the right of the curves $\lambda(\xi)$ solving \eqref{disper_rel2}, equation \eqref{char_equation} has two solutions with  positive real part and  two solutions with negative real part.
\end{proposition}
\begin{proof}
Again, in order to simplify notation, we shall drop the superscript of $U$.

Let  $\lambda \in \mathbb{R}$, $\lambda \gg 1$. The Discriminant of \eqref{char_equation} is
\begin{equation*}
\Delta = \frac{512}{k^6}\Big{(} \frac{\mu^2}{k^2}-2\Big{)}^2\lambda^6+\mathcal{O}(\lambda^5).
\end{equation*}
Consider the depressed quartic equation, associated to \eqref{char_equation}:
\begin{equation*}
z^4 + c_2 z^2 + c_1z + c_0 = 0,
\end{equation*}
which is obtained from \eqref{char_equation} by the change of variable $\nu = z - a_3/4$, where $a_3$ is the third order coefficient of \eqref{char_equation}.
Since $\mu^2 \neq 2 k^2$, we have $\Delta > 0$. Moreover,
\begin{equation*}
c_2 = -\frac{2 \mu}{k^2} \lambda + \mathcal{O}(1),
\end{equation*}
and therefore $c_2<0$.
With the notation 
\begin{equation*}
D=64 \Big{(}c_0-\frac{c_2^2}{4}\Big{)},
\end{equation*}
the following holds:
\begin{itemize}
\item[(i)] if $D<0$, since $c_2<0$, then the roots of \eqref{char_equation} are real and distinct;\\
\item[(ii)] if $D>0$, then \eqref{char_equation} has two pairs of (non-real) complex conjugated roots. 
\end{itemize}
Since 
\begin{equation*}
D =\frac{64(2 k^2 - \mu^2)}{k^4}\lambda^2 + \mathcal{O}(\lambda),
\end{equation*}
then we are in case (i) (resp.\ case (ii)) for $\mu^2>2 k^2$ (resp.\ $\mu^2<2 k^2$).

Assume $\mu^2>2 k^2$. 
We will apply the Descartes' rule of signs to determine the signs of the four real roots of \eqref{char_equation}.  
Since $\lambda \gg 1$, then, disregarding the sign of the term $U-s$, the number of sign changes between consecutive coefficients is two, hence there are at most 2 positive roots. If we substitute $\nu$  with $-\nu$, then there are again two sign changes, so there are at most two negative roots. Since all roots are real, we can conclude that in that case the characteristic equation has two positive and two negative roots.

Now, consider the case $\mu^2<2 k^2$.
On the right of $\lambda(\xi)$, \eqref{char_equation} does not have a purely imaginary root.
Moreover, the second order coefficient of that equation is negative in the regime $\lambda \gg 1$ and the leading coefficient is equal to $1$. Hence, the roots of the equation can not be all in the left half-plane, as in this case the coefficient should be all positive.
Indeed, in that case the equation can be written as a product of linear factors $\nu-\nu_0$, with $\nu_0<0$, which correspond to real roots, and quadratic factors $(\nu-a)^2 + b^2$, which correspond to complex conjugated roots $a\pm i b$, with $a<0$. Each of these factors has positive coefficients, hence the equation  has only positive coefficients.
Moreover, with the substitution  $\nu \rightarrow -\nu$, the second order coefficient is still negative, hence there are roots also in the left half-plane. In conclusion,  there are two complex conjugate roots in the left half-plane and two complex conjugate roots in the right half-plane and the proof is complete.
%
\end{proof}
\section{Analysis of point spectrum}\label{point_spectrum}
For the analysis of the  point spectrum of our linearized operator around the profile, namely to locate its eigenvalues solving the problem  \eqref{eq_variable_coeff}, we shall use  the Evans function, as eigenvalues are zeros of the latter. To this end, we need to be in the situation of  consistent splitting and therefore in the sequel we shall assume 
$\mu^2\neq 2k^2$; see Proposition \ref{prop:2}.
\subsection{System in integrated variables}\label{subsec:intvar}
In order to remove the zero eigenvalue, which is always present, being the corresponding eigenfunction given by the derivative of the profile, without further modifications of the spectrum  \cite{Humpherys},  we re--express the above linearized systems in terms of integrated variables. To this end, we consider the integrated variables
\begin{equation*}
\hat{\rho}(x)=\int_{-\infty}^x\rho(y)dy,\mbox{ }\hat{u}(x)=\int_{-\infty}^xu(y)dy
\end{equation*}
and we  rewrite the eigenvalue equation \eqref{eq_variable_coeff} as a first order system as follows:
\begin{equation}
\label{first_order_system_variable_c}
V' = M(x,\lambda)V,
\end{equation}
for
\begin{equation*}
M(x,\lambda)= \begin{bmatrix}
0 & 0 & 1 & 0\\
-\frac{\lambda+U'(x)}{R(x)} & -\frac{R'(x)}{R(x)} & \frac{s-U(x)}{R(x)} & 0\\
0 & 0 & 0 & 1\\
m_{4,1} &m_{4,2}  &m_{4,3}  &m_{4,4}
\end{bmatrix},
\end{equation*}
where, as before,  $V=[\rho,u,u_1,u_2]^T$, $\rho' = u_1$ and $u_1' = u_2$, and
\begin{align*}
m_{4,1} &= \frac{2}{k^2}\bigg{(} (f_2+g_2)(U' + \lambda) - R(f_1'+g_3+f_3) + \mu \Big{(} U'' - \frac{2 R' U'}{R} - \frac{2 R' \lambda}{R} \Big{)}\bigg{)},\\
m_{4,2} &= \frac{2}{k^2}\bigg{(} R(\lambda - f_2' - g_1) + R'(f_2+g_2) + \mu \Big{(} R'' - \frac{2 (R')^2}{R} \Big{)} \bigg{)},\\
m_{4,3} &= \frac{2}{k^2}\bigg{(} -R(f_1+f_4+g_4) - f_2(f_2+g_2) +\mu \Big{(}\frac{2 f_2 R'}{R} - f_2' + U' + \lambda \Big{)} \bigg{)},\\
m_{4,4} &= -\frac{2}{k^2}\bigg{(} R(f_5+g_5) + \mu f_2 \bigg{)},
\end{align*}
with the following notations: 
\begin{align*}
f_1(x) &= -\frac{dh}{dR}(R(x)), \nonumber\\
f_2(x) &= s - U(x), \nonumber\\
f_3(x) &= \frac{k^2}{2}\big{(}(R(x)^{-\frac{1}{2}}(R(x)^{-\frac{1}{2}})'')' - (R(x)^{-\frac{3}{2}}(R(x)^\frac{1}{2})'')' \big{)}, \nonumber\\
f_4(x) &= \frac{k^2}{2}\big{(}R(x)^{-\frac{1}{2}}(R(x)^{-\frac{1}{2}})'' + 2(R(x)^{-\frac{1}{2}}(R(x)^{-\frac{1}{2}})')' - R(x)^{-\frac{3}{2}}(R(x)^\frac{1}{2})''\big{)}, \nonumber\\
f_5(x) &= -k^2\frac{R'(x)}{R(x)^2}, \nonumber\\
g_1(x) &= \mu \Big{(} \frac{R(x)'}{R(x)} \Big{)}', \nonumber\\
g_2(x) &= \mu \frac{R'(x)}{R(x)}, \nonumber\\
g_3(x) &= \mu \Big{(} \Big{(}\frac{U'(x)}{R(x)}\Big{)}' - \Big{(}\frac{(R(x)U(x))'}{R(x)^2}\Big{)}' \Big{)}, \nonumber\\
g_4(x) &= \mu \Big{(} \frac{U'(x)}{R(x)} + \Big{(}\frac{U(x)}{R(x)}\Big{)}' - \frac{(R(x)U(x))'}{R(x)^2} \Big{)}, \nonumber\\
g_5(x) &= \mu \frac{U(x)}{R(x)}.
\end{align*}
Since the profile $[R(x), U(x)]$ converges as $x \rightarrow \pm \infty$, the system \eqref{first_order_system_variable_c} has exponential dichotomies on  $\mathbb{R}_0^+$ and $\mathbb{R}_0^-$.
Let $S^{\pm}$ be the subspaces of initial conditions that decay exponentially as $x \rightarrow \pm \infty$. Since any eigenfunction $V(x)$ is bounded and solves \eqref{first_order_system_variable_c}, $V(0)$ lies in $S^+$ and in $S^-$ and therefore, $\rho(x)$ and $u(x)$ decay exponentially as $|x|\rightarrow +\infty$.

For $\lambda \neq 0$, integrating \eqref{eq_variable_coeff} yields
\begin{equation*}
\int \rho dx = 0, \mbox{ } \int u dx = 0.
\end{equation*}
In addition, we will show that $\hat{\rho}(x)$ and $\hat{u}(x)$ decay exponentially as $|x| \rightarrow +\infty$. Let us consider the case for $x \rightarrow +\infty$; the other cases being similar.
Since, in particular, we have $|\rho(x)| \leq C_1 \exp(-C_2 x)$, we get
\begin{align*}
|\hat{\rho}(x)|&= \left \lvert \int_{x}^\infty \rho(y)dy \right \rvert  \leq \int_{x}^\infty |\rho(y)|dy \leq C_1 \int_{x}^\infty \exp(-C_2 y) dy \\ &=\frac{C_1}{C_2}\exp(-C_2 x).
\end{align*}
Hence, $\hat{\rho}(x)$ decays exponentially as $x \rightarrow +\infty$.

Expressing $\rho$ and $u$ in terms of $\hat{\rho}$ and $\hat{u}$ and integrating \eqref{eq_variable_coeff} from $-\infty$ to $x$ we get the system in integrated variables:
\begin{align}
\lambda \hat{\rho}&=f_2\hat{\rho}'-R\hat{u}', \label{sys_int_vars1}\\
\lambda \hat{u}&= f_1 \hat{\rho}'+f_2 \hat{u}' + \mu \big{(}R^{-1}(R \hat{u}'+U\hat{\rho}')'-R^{-2}(RU)'\hat{\rho}'\big{)}\nonumber\\
&+\frac{k^2}{2}\big{(}R^{-\frac{1}{2}}(R^{-\frac{1}{2}}\hat{\rho}')''-R^{-\frac{3}{2}}(R^\frac{1}{2})''\hat{\rho}'\big{)}, \label{sys_int_vars2}
\end{align}
where
\begin{align*}
f_1(x)&= -\frac{dh}{dR}(R(x)) \nonumber,\\
f_2(x)&= s-U(x).
\end{align*}
Correspondingly, the system \eqref{sys_int_vars1}-\eqref{sys_int_vars2} can be rewritten as
\begin{equation*}
\hat{V}'=\hat{M}(x,\lambda)\hat{V},
\end{equation*}
where $\hat{V}=[\hat{\rho},\hat{u},\hat{u}_1,\hat{u}_2]^T$,  with $\hat{\rho}' = \hat{u}_1$ and $\hat{u}_1' = \hat{u}_2$, and
\begin{equation}
\label{mat_M1_c}
\hat{M}(x,\lambda)= \begin{bmatrix}
0 & 0 & 1 & 0\\
-\frac{\lambda}{R} & 0 & \frac{f_2}{R} & 0\\
0 & 0 & 0 & 1\\
\hat{m}_{4,1} & \hat{m}_{4,2} & \hat{m}_{4,3}& \hat{m}_{4,4}
\end{bmatrix},
\end{equation}
with
\begin{align*}
\hat{m}_{4,1} &= \frac{2 \lambda f_2}{k^2},\nonumber\\
\hat{m}_{4,2} &= \frac{2 \lambda R}{k^2},\nonumber\\
\hat{m}_{4,3} &= \frac{2}{k^2}\bigg{(}-R f_1 -f_2^2+\mu\Big{(}\frac{(RU)'}{R} + \lambda\Big{)}\bigg{)} + \frac{R''}{R}-\frac{(R')^2}{R^2},\nonumber\\
\hat{m}_{4,4} &= -\frac{2 \mu s}{k^2} + \frac{R'}{R}.
\end{align*}
\subsection{The Evans function for large $|\lambda|$}
To define  the Evans function, let us consider the equation $Y'=\hat{M}(y,\lambda)Y$, where $\hat{M}(y,\lambda)$ is defined in \eqref{mat_M1_c}.
As it is manifest, its limits at $\pm\infty$ are given by  the matrices $M^{\pm}$, defined by \eqref{mat_Mn}, corresponding to limit states $P^{\pm}$, and we assume  these matrices are hyperbolic. This is always true if we are to the right of the bound for the essential spectrum; see Proposition \ref{prop:2}. In addition, we assume that $M^{-}$ has $k$ unstable eigenvalues $\nu^{-}_1,\dots,\nu^{-}_k$ (i.e.\ $\Re(\nu^{-}_i)>0$), and $M^{+}$ has $n-k$ stable eigenvalues $\nu^{+}_1,...,\nu^{+}_{n-k}$ (i.e.\ $\Re(\nu^{+}_i)<0$), and denote the corresponding (normalized) eigenvectors by $v^{\pm}_i$. In our case $n=4$ and $k=2$. Let $Y^{-}_i$ be a solution of $Y'=M(y,\lambda)Y$, satisfying $\exp(\nu^{-}_i y)Y^{-}(y)$ tends to $v^{-}_i$ as $y \rightarrow -\infty$ and $\exp(\nu^{+}_i y)Y^{+}(y)$ tends to $v^{+}_i$ as $y \rightarrow +\infty$. Then,  the Evans function can be defined by
\begin{equation*}
E(\lambda) = \det(Y^-_1(0), .., Y^-_k(0),  Y^+_1(0), ... ,Y^+_{n-k}(0)).
\end{equation*}
As a consequence, $\lambda$ is in the point spectrum of $L$ if and only if $E(\lambda)=0$.

Now, to analyze the behavior of the Evans function for large $\lambda$, let us start by recalling the eigenvalue problem \eqref{eq_variable_coeff}:
\begin{align}
\lambda \rho &= s \rho' - (R u + U \rho)' \label{eigenvalue_c_1},\\
\lambda u &= (f_1 \rho)' + (f_2 u)' \nonumber\\
&+ \mu u'' + g_1 u + g_2 u' + g_3\rho + g_4 \rho' + g_5 \rho'' \nonumber\\
&+ f_3\rho + f_4 \rho' + f_5 \rho'' + \frac{k^2}{2 R} \rho''' \label{eigenvalue_c_2},
\end{align}
where $f_i$ and $g_i$, $i=1,\dots,5$ are defined above. 
Integrating equation \eqref{eigenvalue_c_1} from $-\infty$ to $x$ and expressing $\rho$ in terms of the integrated variable $\hat{\rho}$, we get:
\begin{equation}
\label{expr_u}
\lambda \hat{\rho} = (s - U) \hat{\rho}' - R u.
\end{equation}
We solve the above equation  for $u$ and substitute in \eqref{eigenvalue_c_2} to end up to the following scalar equation:
\begin{align}
&\hat{\rho}^{(4)}+2 \Big {(}\frac{s \mu}{k^2} - \frac{R'}{R} \Big{)} \hat{\rho}''' \nonumber\\
&+\frac{2}{k^2}\big{(} R(f_1+f_4+2\mu f_6'+g_4) + (s-U)(f_2+g_2) - \lambda \mu \big{)} \hat{\rho}'' \nonumber\\
&+\frac{2}{k^2} \Big{(} R(f_1' + f_3 + (f_2 + g_2)f_6'+\mu f_6''+g_3)
+(s-U)(f_2'+g_1) \nonumber\\
&+\lambda\Big{(}2(U-s)+\mu \frac{R'}{R}\Big{)}\Big{)}\hat{\rho}' \nonumber\\
&+\frac{2}{k^2} \Big{(} \lambda \Big{(} -f_2' - g_1 + (f_2 + g_2)\frac{R'}{R} - \mu R (R^{-1})''\Big{)} + \lambda^2 \Big{)} \hat{\rho} = 0, \label{scalar_eq_c}
\end{align}
where
\begin{align*}
f_6(x)=\frac{s-U(x)}{R(x)}.
\end{align*}
\begin{lemma}
$\lambda \neq 0$ is an eigenvalue for \eqref{eigenvalue_c_1}-\eqref{eigenvalue_c_2} if and only if it is an eigenvalue for \eqref{scalar_eq_c}. The Evans function for \eqref{scalar_eq_c} does not vanish for $\Re(\lambda) \geq 0$ and $|\lambda|$ large enough.
\end{lemma}
\begin{proof}
As said before, for $\lambda \neq 0$ we integrate  \eqref{eigenvalue_c_1}-\eqref{eigenvalue_c_2} to obtain
\begin{equation*}
\int \rho dx = 0, \mbox{ } \int u dx = 0.
\end{equation*}
We are thus allowed to use  the integrated variable
\begin{equation*}
\hat{\rho}(x) = \int_{-\infty}^x \rho(y) dy,
\end{equation*}
which decays exponentially as $|x| \rightarrow + \infty$. Thus, from \eqref{expr_u} we get
\begin{equation*}
u = -\frac{\lambda}{R}\hat{\rho} + \frac{s-U}{R}\hat{\rho}'
\end{equation*}
and hence \eqref{scalar_eq_c}.
In particular, if $\lambda \neq 0$ is not an eigenvalue of \eqref{scalar_eq_c}, it is also not an eigenvalue of \eqref{eigenvalue_c_1}-\eqref{eigenvalue_c_2}.

Now, we make a change of variable
\begin{equation*}
x=\frac{y}{|\lambda|^\frac{1}{2}}
\end{equation*}
and, dividing \eqref{scalar_eq_c} by $|\lambda|^2$ yields
\begin{align}
&\frac{d^4\hat{\rho}}{dy^4}+\frac{2}{|\lambda|^\frac{1}{2}} \Big {(}\frac{s \mu}{k^2} - \frac{R'}{R} \Big{)} \frac{d^3\hat{\rho}}{dy^3} \nonumber\\
&+\frac{2}{k^2}\Big{(} \frac{R(f_1+f_4+2\mu f_6'+g_4) + (s-U)(f_2+g_2)}{|\lambda|} - \frac{\lambda}{|\lambda|} \mu \Big{)} \frac{d^2 \hat{\rho}}{dy^2} \nonumber\\
&+\frac{2}{k^2 |\lambda|^\frac{3}{2}} \Big{(} R(f_1' + f_3 + (f_2 + g_2)f_6'+\mu f_6''+g_3)
+(s-U)(f_2'+g_1) \nonumber\\
&+\lambda\Big{(}2(U-s)+\mu \frac{R'}{R}\Big{)}\Big{)}\frac{d\hat{\rho}}{dy} \nonumber\\
&+\frac{2}{k^2} \Big{(} \frac{\lambda}{|\lambda|^2} \Big{(} -f_2' - g_1 + (f_2 + g_2)\frac{R'}{R} - \mu R (R^{-1})''\Big{)} + \frac{\lambda^2}{|\lambda|^2} \Big{)} \hat{\rho} = 0. \label{rescaled_equation_c}
\end{align}
 Taking the limit $|\lambda|\rightarrow + \infty$ in \eqref{rescaled_equation_c} we end up with
\begin{equation}
\label{constant_coeff_equation_c}
\frac{d^4 \hat{\rho}}{dy^4} - \frac{2 \mu}{k^2}\tilde{\lambda} \frac{d^2 \hat{\rho}}{dy^2} +\frac{2}{k^2}\tilde{\lambda}^2 \hat{\rho}=0,
\end{equation}
where $\tilde{\lambda}={\lambda}/{|\lambda|}$.
The equation \eqref{constant_coeff_equation_c} can be rewritten in a standard way as a first-order system as follows:
\begin{equation}
\label{first_order_constant_coefficient}
\frac{d}{dy}\begin{bmatrix}
\rho_1\\ \rho_2\\ \rho_3 \\ \rho_4
\end{bmatrix}
 = \begin{bmatrix}
0 & 1 & 0 & 0\\
0 & 0 & 1 & 0\\
0 & 0 & 0 & 1\\
-\frac{2}{k^2}\tilde{\lambda}^2 & 0 & \frac{2 \mu}{k^2} \tilde{\lambda} & 0
\end{bmatrix}
\begin{bmatrix}
\rho_1\\ \rho_2\\ \rho_3 \\ \rho_4
\end{bmatrix},
\end{equation}
with associated  characteristic equation given by
\begin{equation}
\label{char_eq_1}
z^4 - \frac{2 \mu}{k^2}\tilde{\lambda} z^2 + \frac{2 \tilde{\lambda}^2}{k^2}=0.
\end{equation}
Let $\Re(\lambda) \geq 0$. We claim that, under the condition $\mu^2\neq 2 k^2$, 
\eqref{char_eq_1} has four distinct roots, with $\Re(z_1),\Re(z_2)<0$ and $\Re(z_3),\Re(z_4)>0$.
Indeed, since in addition $\tilde{\lambda}\neq 0$, then 
\begin{equation*}
D = \frac{4}{k^2}\Big{(} \frac{\mu^2}{k^2}-2\Big{)}\tilde{\lambda}^2 \neq 0.
\end{equation*}
%
%
Then, we make the change of variable $w=z^2$ to rewrite \eqref{char_eq_1} as follows:
\begin{equation}
\label{quadratic_char_eq}
w^2 - \frac{2 \mu}{k^2}\tilde{\lambda} w + \frac{2 \tilde{\lambda}^2}{k^2}=0.
\end{equation}
If $\Re(\lambda) \geq 0$, then $\Re(\tilde{\lambda}) \geq 0$ and, since $D \neq 0$, and in particular $\tilde{\lambda}\neq 0$, 
the equation \eqref{quadratic_char_eq} has two distinct nonzero roots $w_{1,2}$. Hence, the four roots $z_i$, $i=1,\dots 4$, of \eqref{char_eq_1} are distinct as well.

\textit{Case 1: $\mu^2>2k^2$ --- viscosity dominant case.}
In that case, $w_{1,2}$ are given by:
\begin{equation}
\label{eq:diffdom}
w_{1,2} = \left (\frac{\mu}{k^2} \pm \frac{1}{k}\sqrt{\frac{\mu^2}{k^2}-2}\right )\tilde{\lambda} = w_\pm \tilde{\lambda}
\end{equation}
for $w_+>\mu /k^2>w_->0$ real, positive numbers. Hence, $w_{1,2}$ are not real negative, so that the four solutions of $z^2 = w_j$, $j=1,2$, are not purely imaginary and, more precisely,  $z^2 = w_j$ has one solution with positive real part and one with negative real part, for $j=1,2$. 

\textit{Case 2: $\mu^2<2k^2$ --- dispersion dominant case.}
In that case, $w_{1,2}$  are given by:
\begin{equation}
\label{eq:dispdom}
w_{1,2} = \left (\frac{\mu}{k^2} \pm \frac{i}{k}\sqrt{2-\frac{\mu^2}{k^2}}\right )\tilde{\lambda}.
\end{equation}
and $|w_j| = \sqrt{2}/k$, $j=1,2$. Let $\tilde\lambda = \exp(i \theta)$, and $\theta \in [0, \pi/2] \cup [3 \pi/2, 2 \pi[$. We have
\begin{equation*}
\mathop{Arg}\left (\frac{\mu}{k^2} + \frac{i}{k}\sqrt{2-\frac{\mu^2}{k^2}}\right ) \in ]0, \pi/2[.
\end{equation*}
If $\theta \in [0, \pi/2]$, then $\mathop{Arg}(w_1) \in ]0, \pi[$; if $\theta \in [3 \pi/2, 2 \pi[$, then $\Re(w_1) > 0$. In both cases,  $w_1$ is not real negative. Analogously, 
\begin{equation*}
\mathop{Arg}\left (\frac{\mu}{k^2} - \frac{i}{k}\sqrt{2-\frac{\mu^2}{k^2}}\right ) \in ]3 \pi/2, 2 \pi[.
\end{equation*} 
If $\theta \in [0, \pi/2]$, then $\Re(w_2) > 0$; if $\theta \in [3 \pi/2, 2 \pi[$, then $\mathop{Arg}(w_1) \in ]\pi, 2 \pi[$. Again, in both cases  $w_{2}$ is not real negative and we conclude as before.

The equation \eqref{constant_coeff_equation_c} has constant coefficients, so its Evans function can be computed explicitly as follows.
Let $z_i$  be a simple eigenvalue of the matrix
\begin{equation}
\label{matrix_first_order_constant_coefficient}
 \begin{bmatrix}
0 & 1 & 0 & 0\\
0 & 0 & 1 & 0\\
0 & 0 & 0 & 1\\
-\frac{2}{k^2}\tilde{\lambda}^2 & 0 & \frac{2 \mu}{k^2} \tilde{\lambda} & 0
\end{bmatrix}
\end{equation}
with associated eigenvector $v_i = [1,z_i,z_i^2,z_i^3]^T$. Then, the Evans function 
is given by
\begin{align*}
\tilde{E}(\lambda) &= \det([v_1, v_2,v_3,v_4]) = \prod_{j<k}(z_j-z_k) \\
&=(z_1-z_2)(z_1-z_3)(z_1-z_4)(z_2-z_3)(z_2-z_4)(z_3-z_4) \\
&\neq 0,
\end{align*}
because the eigenvalues $z_i$ are distinct.

Now, as already done for \eqref{constant_coeff_equation_c}, we rewrite \eqref{rescaled_equation_c} in a standard way  as a first order system as follows:
\begin{equation}
\label{rescaled_first_order_sys}
W'=A(y,\lambda)W.
\end{equation}
The matrix \eqref{matrix_first_order_constant_coefficient} is hyperbolic, so the system \eqref{first_order_constant_coefficient} has exponential dichotomies on $\mathbb{R}_0^+$ and $\mathbb{R}_0^-$. 
For sufficiently large $|\lambda|$, the coefficients of \eqref{first_order_constant_coefficient} and \eqref{rescaled_first_order_sys} are  close to each other, uniformly in $y$. Thus,  from \cite[Theorem 3.1]{Sandstede} it follows that the system \eqref{rescaled_first_order_sys} also has exponential dichotomies and, moreover, the projections corresponding to \eqref{first_order_constant_coefficient} are close to the ones corresponding to \eqref{rescaled_first_order_sys}. So, the Evans functions of \eqref{constant_coeff_equation_c} and 
\eqref{rescaled_equation_c} are uniformly close in $\lambda$. Therefore the Evans function for \eqref{scalar_eq_c} never vanishes for $\Re(\lambda)\geq 0$ and $|\lambda|>C$, where $C$ is some (sufficiently big) constant. 
\end{proof}
In the previous lemma we proved that, if $\lambda$ is an  eigenvalue of \eqref{scalar_eq_c}  with $\Re(\lambda)\geq 0$, then we must  have $|\lambda| \leq C$, for a constant $C$ sufficiently big. 
In the next sections we shall obtain a \emph{quantitative} estimate for that constant $C$ to be able to analyze numerically the behaviour of the Evans function on $\{\Re(\lambda)\geq0,|\lambda| < C\}$.
\subsection{Estimate for the maximum of $|\lambda|$}
In this section we decompose the system into a constant coefficients part, which depends only on the direction $\lambda/|\lambda|$, and a perturbation, which becomes small for large $|\lambda|$. Then, we use exponential dichotomies to estimate the difference between the Evans functions of the constant coefficient system and the perturbed system.

To this end, let us consider
\begin{equation}
\label{system_constant_coefficients1}
\frac{du}{dx}=A(\tilde{\lambda})u
\end{equation}
where $\tilde{\lambda} = \frac{\lambda}{|\lambda|}$ and $A(\tilde{\lambda})$ is defined in \eqref{matrix_first_order_constant_coefficient}. The matrix $A(\tilde{\lambda})$ does not depend on $x$ and it has simple eigenvalues $z_j$, $j=1,...,4$, with $\Re(z_1), \Re(z_2) < 0$ and $\Re(z_3), \Re(z_4) > 0$ and we may consider $\tilde{\lambda}$ fixed. The system \eqref{system_constant_coefficients1} has an exponential dichotomy (see \cite{Coppel}, Chapter 4) on $\mathbb{R}_{0}^+$, namely,   there are positive constants $K$, $\alpha$ and projection $P$ such that
\begin{align*}
\|X(x)PX^{-1}(y)\|_2 \leq K e^{-\alpha (x-y)}, x \geq y &\geq 0,\\
\|X(x)(Id-P)X^{-1}(y)\|_2 \leq K e^{-\alpha (y-x)}, y \geq x &\geq 0,
\end{align*}
where $X(x)$ is the fundamental solutions matrix for \eqref{system_constant_coefficients1} with $X(0)=Id$. We introduce  the usual notations for   the scalar product $\langle v,w\rangle= v \cdot \bar{w}$, the vector norm  $|v| = \sqrt{\langle v, v \rangle}$,  and the 2-norm  $\|A\|_2 = \sup_{|v|=1}|A v|$. Moreover, for later use, let us also introduce the norm $\|A\|_F = \sqrt{\sum_{j,k} |a_{j,k}|^2}$ and recall the inequality $\|A\|_2 \leq \|A\|_F$ holds. 
Since the matrix $A(\tilde{\lambda})$ has constant coefficients, the constants $K$ and $\alpha$ can be explicitly computed.

Now, we rewrite  system \eqref{rescaled_first_order_sys} as the following  perturbed system
\begin{equation}
\label{perturbed_system}
\frac{du}{dx}=A(\tilde{\lambda})u + B(x,\lambda) u,
\end{equation}
where
\begin{equation*}
\lim_{| \lambda | \rightarrow +\infty}\sup_{x \geq 0} \|B(x,\lambda)\|_2 = 0, 
\end{equation*}
and, more precisely,  
\begin{equation*}
\sup_{x \geq 0} \|B(x,\lambda)\|_2 = \mathcal{O}(|\lambda|^{-\frac{1}{2}}). 
\end{equation*}
Denote $\delta = \sup_{x \geq 0} \|B(x,\lambda)\|_2$. If $\delta < \alpha/(4 K^2)$, then the perturbed system \eqref{perturbed_system} also has an exponential dichotomy with projection $Q$ and, moreover 
\begin{equation*}
\|P-Q\|_2 \leq 4 \alpha^{-1} K^3 \delta;
\end{equation*}
see \cite[Chapter 4, Proposition 1]{Coppel} for details. 
Denoting with $M$ and $N$ the 2--dimensional subspaces related to the projections $P$ and $Q$, thanks to  \cite[page 58, Theorem 6.35]{Kato}, there exist unique \emph{orthogonal} projections $\tilde{P}$ and $\tilde{Q}$ onto $M$ and $N$ which verify 
\begin{equation*}
\| \tilde{P}-\tilde{Q} \|_2 \leq \| P - Q \|_2\leq \epsilon,
\end{equation*}
for $\epsilon = 4 \alpha^{-1} K^3 \delta$. 

Let $v_j$ be the eigenvectors of $A(\tilde{\lambda})$ related to the stable eigenvalues $z_j$, $j=1,2$,  and normalized so that $|v_j| = 1$. Then 
we have $v_j \in M$, that is $\tilde{P} v_j = v_j$. Denoting $h_j = v_j - \tilde{Q} v_j$, we have
\begin{equation*}
|h_j| =|v_j - \tilde{Q} v_j | = | \tilde{P} v_j - \tilde{Q} v_j | \leq \| \tilde{P} - \tilde{Q}\|_2 |v_j | \leq \epsilon.
\end{equation*}
Still for $j=1,2$, let  us define $\tilde{v}_j = \tilde{Q} v_j$. Then $\langle \tilde{v}_j, h_j \rangle = \langle \tilde{Q} v_j, v_j - \tilde{Q} v_j \rangle = 0$. Hence $|v_j|^2 = |\tilde{v}_j|^2 + |h_j|^2$ and  therefore $|\tilde{v}_j|^2 \geq |v_j|^2 - \epsilon^2 = 1-\epsilon^2$, as well as $|\tilde{v}_1||\tilde{v}_2| \geq  1-\epsilon^2$. As a consequence,  there exists   $\epsilon_0>0$ such that  for any $ \epsilon < \epsilon_0$, $\tilde{v}_j \neq 0$. Also, 
\begin{equation*}
 |\langle \tilde{v}_1,  \tilde{v}_2 \rangle|  =   |\langle v_1 - h_1,  v_2 - h_2 \rangle| 
 \leq |\langle v_1, v_2 \rangle| + 2 \epsilon + \epsilon^2.
 \end{equation*}
Hence,  since $v_1$ and $v_2$ are linearly independent unit vectors, $|\langle v_1, v_2 \rangle| <1$ and we can further choose $\epsilon_0>0$ such that, if $\epsilon < \epsilon_0$, 
 \begin{equation*} 
 |\tilde{v}_1||\tilde{v}_2| \geq 1-\epsilon^2 > |\langle v_1, v_2 \rangle| + 2 \epsilon + \epsilon^2 \geq |\langle \tilde{v}_1,  \tilde{v}_2 \rangle|.
 \end{equation*}
 Finally,  as $\tilde{v}_1$ and $\tilde{v}_2$ verify strict  Cauchy--Schwarz inequality, they are linearly independent and thus $\{\tilde{v}_1,\tilde{v}_2\}$ is a basis of $N$, namely $N = span(\{\tilde{v}_1,\tilde{v}_2\})$.
Referring to  $\mathbb{R}_0^-$, we  argue in an analogous way to obtain the vectors $\tilde{v}_3$ and $\tilde{v}_4$ needed to compute the Evans function we are looking for. 
Indeed, denoting with  $E(\lambda)$ the Evans function for \eqref{system_constant_coefficients1}, we have   $E(\lambda) = \det([v_1,v_2,v_3,v_4])$. Moreover, if $E_p(\lambda)$ denotes the Evans function for \eqref{perturbed_system}, then $E_p(\lambda) = \det([\tilde{v}_1,\tilde{v}_2,\tilde{v}_3,\tilde{v}_4])$. 

In what follows, we shall obtain sufficient conditions to (numerically) conclude that  $E_p(\lambda)\neq 0$ in the region $|\lambda|\geq C$ of the unstable half-plane. This will be obtained  by proving that 0 can not be an eigenvalue of the matrix $V-H$, where $H = [h_1,...,h_4]$, $V = [v_1,...,v_4]$, using  the Bauer--Fike Theorem \cite{Golub}.
To this end, let us first diagonalize system \eqref{perturbed_system} as follows:
\begin{equation}
\label{diagonal_perturbed_system}
\frac{dv}{dx} = D(\tilde{\lambda})v + S^{-1}B(x,\lambda)S v,
\end{equation}
for $u = S v$, where we denote with $S$ the matrix of eigenvectors of $A(\tilde{\lambda})$ and  $D(\tilde{\lambda}) = diag(z_1,z_2,z_3,z_4) = S^{-1} A(\tilde{\lambda})S $. Since $D(\tilde{\lambda})$ is diagonal, its eigenvectors are given by $e_j$, the standard basis vectors, for $j=1,\dots,4$, and, referring to exponential
dichotomy properties for that diagonalized system, we conclude $P = diag(1,1,0,0)$ and $K = 1$. Moreover, in the notation before,  $V = Id$ and $\det(V) = \|V\|_2=\|V^{-1}\|_2=1$. 
Now, 
let $q$ be an eigenvalue of $V-H$. Hence, a direct application of the Bauer--Fike Theorem implies that
\begin{equation*}
|1-q|\leq\|H\|_2.
\end{equation*}
Since $\|H\|_2 \leq \|H\|_F$, $\|H\|_F<1$ implies in particular  that $0$ can not be an eigenvalue for $V-H$, that is $\det(V-H)\neq 0$, namely,  $E_p(\lambda)\neq 0$.
In next sections we shall prove the above estimate for $|\lambda| \geq C$, with $C$ \emph{explicit}.

To this end, we can directly compute $S^{-1}B(x,\lambda)S$ and obtain an explicit bound for its norms.  For this, let  $\delta_\pm$ (depending on $|\lambda|$) be upper bounds for the norm $\|S^{-1}B(x,\lambda)S\|_F$ on $\mathbb{R}_0^{\pm}$, so that
\begin{equation*}
\|S^{-1}B(x,\lambda)S\|_2 \leq \|S^{-1}B(x,\lambda)S\|_F \leq \delta_\pm,\mbox{ }x \in \mathbb{R}_0^{\pm}.
\end{equation*}
Moreover, denote
\begin{equation}
\epsilon_{\pm} = 4 \alpha^{-1} \delta_\pm
\label{def_eps}
\end{equation}
and consider the condition
\begin{equation}
\sqrt{2}\sqrt{\epsilon_{+}^2+\epsilon_{-}^2}<1\label{cond_epsilon}.
\end{equation}
Clearly, \eqref{cond_epsilon} in particular implies 
$\epsilon_\pm^2 < 1/2$ and, from the definition of $\epsilon_{\pm}$ in \eqref{def_eps}, we readily obtain $\delta_{\pm} < \alpha/4$.
The latter condition on $\delta_{\pm}$ guarantees the existence of exponential dichotomies on $\mathbb{R}_0^\pm$
with  the properties stated above,
which  implies $|h_j| \leq \epsilon_+$, for $j=1,2$, and $|h_j| \leq \epsilon_-$, for $j=3,4$. Hence, we obtain 
\begin{equation*}
\| H \|_F = \sqrt{|h_1|^2 + ... + |h_4|^2} \leq \sqrt{2 \epsilon_+^2 + 2 \epsilon_-^2} = \sqrt{2}\sqrt{\epsilon_+^2 + \epsilon_-^2}< 1.
\end{equation*}
Therefore, if \eqref{cond_epsilon} holds, then we have $\| H \|_F<1$, which shows that the Evans function for \eqref{diagonal_perturbed_system} does not vanish in the value $\lambda$ under consideration, that is $E_p(\lambda) \neq 0$.

The final result we are interesting in, that is the fact that the Evans function for \eqref{scalar_eq_c} does not vanish, is a consequence of the condition $E_p(\lambda) \neq 0$, after a change of varible, which for completeness we shall present here below.

Let   $u \in \mathbb{C}^n$ by a solution of a general system of ODEs:
\begin{equation*}
\frac{du}{dx} = \mathcal{A}(x,\lambda)u, 
\end{equation*}
where $\mathcal{A}(x,\lambda) \in \mathbb{C}^{n \times n}$.
Then, after a change the independent variable $y=c x$, $c>0$, for an invertible matrix $T \in \mathbb{C}^{n \times n}$, $v(y)=T u(y/c)$ solves
\begin{equation*}
\frac{dv}{dy} = \frac{1}{c}T \mathcal{A}\Big{(}\frac{y}{c},\lambda\Big{)}T^{-1}v.
\end{equation*}
Hence, 
let us rewrite \eqref{scalar_eq_c} as a first order system
\begin{equation}
\frac{d\tilde{u}}{dx}=\tilde{A}(x,\lambda)\tilde{u}.\label{first_order_system}
\end{equation}
The rescaled equation \eqref{rescaled_first_order_sys} is obtained from \eqref{first_order_system} by the aforementioned change of variables:
\begin{equation*}
y = |\lambda|^\frac{1}{2}x,\mbox{ }W(y)= \tilde{D} \tilde{u}(y/|\lambda|^{\frac{1}{2}}),
\end{equation*}
where
\begin{equation*}
\tilde{D} = diag(1,|\lambda|^{-\frac{1}{2}},|\lambda|^{-1},|\lambda|^{-\frac{3}{2}}).
\end{equation*}
Moreover, the diagonalized system \eqref{diagonal_perturbed_system} is obtained from \eqref{rescaled_first_order_sys} by a further change of the unknown $v(y)=S^{-1}W(y)$. Suppose now $\tilde{u}^\star(x)$ is an eigenfunction of \eqref{first_order_system}. Then
\begin{equation*}
v^{\star}(x) = S^{-1}\tilde{D}\tilde{u}^{\star}\Big(\frac{x}{|\lambda|^{\frac{1}{2}}}\Big{)}
\end{equation*}
is an eigenfunction of \eqref{diagonal_perturbed_system}, which is impossible if  $E_p(\lambda) \neq 0$. Therefore, there exist no eigenfunctions for \eqref{first_order_system} under this condition, or, equivalently, the Evans function for \eqref{scalar_eq_c} does not vanish, provided $E_p(\lambda) \neq 0$.

Summarizing, to conclude our analysis we shall prove we can find an explicit   constant $C$, such that, for $|\lambda| \geq C$ and $\Re(\lambda)\geq 0$, \eqref{cond_epsilon} is satisfied.
 To this end, in the following sections we shall consider two regimes, namely the viscosity dominant and the dispersion dominant one. In both cases, $\delta_\pm$ can be chosen to be of the form $\sqrt{\tilde{p}(|\lambda|^{-\frac{1}{2}})}$, where $\tilde{p}$ is a polynomial with explicit coefficients. Hence, the condition \eqref{cond_epsilon} involves a function of the same form and we can easily find an explicit constant $C$ such that, for $|\lambda| \geq C$, \eqref{cond_epsilon} is satisfied; see Lemma \ref{upper_bound_lambda_visc} and Lemma \ref{upper_bound_lambda_c} below.
\subsubsection{Estimate for the maximum of $|\lambda|$ --- viscosity dominant case}
In the viscosity dominant case $\mu^2 > 2k^2$, the roots  $z_i$, $i=1,\dots,4$, of the characteristic equation \eqref{char_eq_1} of \eqref{constant_coeff_equation_c} are given by
\begin{equation}
z_{1,3} = \mp\sqrt{w_+}\exp(i \theta/2),\mbox{ } z_{2,4} = \mp\sqrt{w_-}\exp(i \theta/2)\label{roots_visc},
\end{equation}
 where the real, positive  numbers  $w_\pm$ are defined in \eqref{eq:diffdom} and we recall the notation $\tilde\lambda = \exp(i \theta)$.
 
The distances between the roots $|z_j-z_k|$ are
\begin{align*}
|z_1-z_2| &= |z_3-z_4| = \sqrt{w_+} - \sqrt{w_-},\\
|z_1-z_3| &= 2 \sqrt{w_+},\\
|z_1-z_4| &= |z_2-z_3| = \sqrt{w_+} + \sqrt{w_-},\\
|z_2-z_4| &= 2 \sqrt{w_-}.
\end{align*}
In preparation to stating Lemma \ref{upper_bound_lambda_visc} let us introduce the notation
\begin{equation}\label{bounds_k}
\begin{aligned}
m_1(\lambda) &= \frac{2}{k^2}|\lambda|^{-1} \sup_{x \geq 0} \left\lvert  -f_2' - g_1 + (f_2 + g_2)\frac{R'}{R} - \mu R (R^{-1})'' \right\rvert,\\
m_2(\lambda) &= \frac{2}{k^2}\Big{(}|\lambda|^{-\frac{3}{2}}  \sup_{x \geq 0} \left\lvert R(f_1' + f_3 + (f_2 + g_2)f_6'+\mu f_6''+g_3) \right. \\
&\ \left. +(s-U)(f_2'+g_1) \right \rvert
+|\lambda|^{-\frac{1}{2}} \sup_{x \geq 0} \left\lvert 2(U-s)+\mu \frac{R'}{R} \right \rvert\Big{)},\\
m_3(\lambda) &= \frac{2}{k^2}|\lambda|^{-1} \sup_{x \geq 0} \left\lvert R(f_1+f_4+2\mu f_6'+g_4) + (s-U)(f_2+g_2) \right \rvert,\\
m_4(\lambda) &= 2|\lambda|^{-\frac{1}{2}} \sup_{x \geq 0} \left\lvert \frac{s \mu}{k^2} - \frac{R'}{R}  \right \rvert.
\end{aligned}
\end{equation}
We have $|z_1|=|z_3|=\sqrt{w_+}$ and $|z_2|=|z_4|=\sqrt{w_-}$. Furthermore,
\begin{align*}
\tilde{b}_{j,k} &= \frac{p_k}{q_j},\\
 p_k &= \sum_{l=1}^4 m_l(\lambda)|z_k|^{l-1}=\begin{dcases}
\sum_{l=1}^4 m_l(\lambda)(w_+)^\frac{l-1}{2}, & k \in \{1,3\}\\
\sum_{l=1}^4 m_l(\lambda)(w_-)^\frac{l-1}{2}, & k \in \{2,4\},
 \end{dcases}
 \end{align*}
 and
 \begin{align*}
q_1 &= |z_1-z_2||z_1-z_3||z_1-z_4|, q_2 = |z_1-z_2||z_2-z_3||z_2-z_4|,\\
q_3 &= |z_1-z_3||z_2-z_3||z_3-z_4|, q_4 = |z_1-z_4||z_2-z_4||z_3-z_4|.
\end{align*}
Hence,
\begin{align*}
q_1 &= q_3 = 2 \sqrt{w_+}(w_+ - w_-),\\
q_2 &= q_4 = 2 \sqrt{w_-}(w_+ - w_-),
\end{align*}
and
\begin{equation}
\tilde{B}_{+} = [\tilde{b}_{j,k}].
\label{residual_matrix_visc}
\end{equation}
Analogously, the matrix $\tilde{B}_{-}$ is constructed in the same way with the suprema in the definition of $m_k(\lambda)$ taken for $x \leq 0$. 
Finally, we consider
\begin{equation*}
\epsilon_{\pm} = 4 \alpha^{-1} \Vert \tilde{B}_{\pm} \Vert_F, 
\end{equation*}
that is, 
we can choose
$\delta_\pm = \|\tilde{B}_{\pm}\|_F$
and, being $\alpha = \min_k |\Re(z_k)|$, in view of \eqref{roots_visc} we can take 
\begin{equation*}
\alpha~=\sqrt{\frac{w_-}{2}}.
\end{equation*}
\begin{lemma}\label{upper_bound_lambda_visc}
Suppose $\mu^2 > 2k^2$. 
Then, we can find a   constant $C>0$ such that   $\sqrt{2}\sqrt{\epsilon_{+}^2+\epsilon_{-}^2}<~1$ for all $\lambda$ with $\Re(\lambda)\geq 0$ and $|\lambda|\geq C$. As a consequence,  the Evans function $E(\lambda)$ for \eqref{scalar_eq_c}   has no zeros in this region.
\end{lemma}
\begin{proof}
We compute directly the matrix $S^{-1}B(x,\lambda)S$ in \eqref{diagonal_perturbed_system} and we see that $\tilde{b}_{j,k}$ are upper bounds for the absolute values of its entries.
The terms $m_k(\lambda)$ in \eqref{bounds_k} are upper bounds  for the terms $b_{4,k}(\lambda)$ of the matrix  $B(x,\lambda)$ from \eqref{perturbed_system}:
   $|b_{4,k}(x,\lambda)|\leq m_k(\lambda)$, while $|z_k|$ and $|z_j-z_k|$ come from $S$ and $S^{-1}$.  Moreover, note that the matrices $\tilde{B}_{\pm}(\lambda)$ from \eqref{residual_matrix_visc} have monotonically decreasing in $|\lambda|$ entries, therefore, from the bound above, we can easily find $C>0$ such that    $\sqrt{2}\sqrt{\epsilon_{+}^2+\epsilon_{-}^2}<~1$ for all $\lambda$ with $\Re(\lambda)\geq 0$ and $|\lambda|\geq C$ and the proof is complete.
\end{proof}
\subsubsection{Estimate for the maximum of $|\lambda|$ --- dispersion dominant case}
For the sake of simplicity, we fix in this section   $k=\sqrt{2}$ so that the dispersion dominant case reduces to $\mu^2 < 4$ and, with the definition in \eqref{eq:dispdom},   
 \begin{equation*}
|w_{1,2}| = \left |\frac{\mu}{2}\pm i \sqrt{1 - \frac{\mu^2}{4}}\right |=1.
\end{equation*}
Moreover, still in preparation to stating Lemma \ref{upper_bound_lambda_c}, let us introduce the following notations:
\begin{equation}\label{eq:rootsdisp}
\begin{aligned}
\theta_{1,2} &= \mathop{Arg}\left (\frac{\mu}{2}\pm i \sqrt{1 - \frac{\mu^2}{4}}\right ),  \\
z_{1,3} &= \mp \exp(i(\theta+\theta_1)/2),\mbox{ }z_{2,4} = \mp \exp(i(\theta+\theta_2)/2),
\end{aligned}
\end{equation}
where, as before, we recall  $\tilde\lambda = \exp(i \theta)$.
The distances $|z_j-z_k|$ between the roots  of the characteristic equation \eqref{char_eq_1} of \eqref{constant_coeff_equation_c} still do not depend on $\theta$ and we can compute them e.g. for $\theta = 0$.\\
Also,
\begin{align}
m_1(\lambda) &= |\lambda|^{-1} \sup_{x \geq 0} \left\lvert  -f_2' - g_1 + (f_2 + g_2)\frac{R'}{R} - \mu R (R^{-1})'' \right\rvert, \nonumber\\
m_2(\lambda) &= |\lambda|^{-\frac{3}{2}}  \sup_{x \geq 0} \left\lvert R(f_1' + f_3 + (f_2 + g_2)f_6'+\mu f_6''+g_3)
+(s-U)(f_2'+g_1) \right \rvert \nonumber\\
&+|\lambda|^{-\frac{1}{2}} \sup_{x \geq 0} \left\lvert 2(U-s)+\mu \frac{R'}{R} \right \rvert, \nonumber\\
m_3(\lambda) &= |\lambda|^{-1} \sup_{x \geq 0} \left\lvert R(f_1+f_4+2\mu f_6'+g_4) + (s-U)(f_2+g_2) \right \rvert, \nonumber\\
m_4(\lambda) &= |\lambda|^{-\frac{1}{2}} \sup_{x \geq 0} \left\lvert s \mu - \frac{2 R'}{R}  \right \rvert.
\label{bounds}
\end{align}
Moreover,
\begin{align}
\tilde{b}_{j,k} &= \frac{p_k}{q_j}, p_k = \sum_{l=1}^4 m_l(\lambda)|z_k|^{l-1}=\sum_{l=1}^4 m_l(\lambda), \nonumber\\
q_1 &= |z_1-z_2||z_1-z_3||z_1-z_4|, q_2 = |z_1-z_2||z_2-z_3||z_2-z_4|, \nonumber\\
q_3 &= |z_1-z_3||z_2-z_3||z_3-z_4|, q_4 = |z_1-z_4||z_2-z_4||z_3-z_4|, \nonumber\\
\tilde{B}_{+} &= [\tilde{b}_{j,k}],
\label{residual_matrix}
\end{align}
As before, the matrix $\tilde{B}_{-}$ has the suprema in the definition of $m_k(\lambda)$ taken for $x \leq 0$
and we define
\begin{equation*}
\epsilon_{\pm} = 4 \alpha^{-1} \Vert \tilde{B}_{\pm} \Vert_F, 
\end{equation*}
that is
$\delta_\pm = \|\tilde{B}_{\pm}\|_F$.
We have $0<\theta_1<\pi/2$, $-\pi/2 < \theta_2 < 0$, $\theta_2 = - \theta_1$, and recall $\tilde{\lambda} = \exp(i \theta)$, $\theta \in [-\pi/2,  \pi/2]$. Then
\begin{equation*}
-\frac{\pi}{4}+\frac{\theta_j}{2} \leq \frac{\theta + \theta_j}{2}\leq \frac{\pi}{4}+\frac{\theta_j}{2},\mbox{ }j=1,2.
\end{equation*}
Hence,
\begin{equation*}
-\frac{\pi}{4} < \frac{\theta + \theta_1}{2} < \frac{\pi}{2},\mbox{ }-\frac{\pi}{2} < \frac{\theta + \theta_2}{2} < \frac{\pi}{4}.
\end{equation*}
Since $\cos (x)$ is positive for $-\pi/2 < x < \pi/2$, in view of \eqref{eq:rootsdisp}, we have
\begin{align*}
|\Re(z_1)| &= |\Re(z_3)| = \cos\Big{(}\frac{\theta+\theta_1}{2}\Big{)},\\
|\Re(z_2)| &= |\Re(z_4)| = \cos\Big{(}\frac{\theta+\theta_2}{2}\Big{)}.
\end{align*}
Moreover,
\begin{equation*}
\alpha = \min_{\theta \in [-\pi/2,\pi/2]} \min_{k \in \{1,...,4\}} |\Re(z_k)|.
\end{equation*}
Since $\cos (x)$ is strictly increasing on $[-\pi/2,0]$, and strictly decreasing on $[0,\pi/2]$, it follows in particular that $\cos((\theta + \theta_1)/2)$ is strictly increasing for $\theta \in [-\pi/2,-\theta_1]$, and strictly decreasing for $\theta \in [-\theta_1,\pi/2]$. Hence,
\begin{align*}
\min_{\theta \in [-\pi/2,-\theta_1]}\cos\Big{(}\frac{\theta+\theta_1}{2}\Big{)} &= \cos\Big{(}-\frac{\pi}{4} + \frac{\theta_1}{2}\Big{)} = \cos\Big{(}\frac{\pi}{4} - \frac{\theta_1}{2}\Big{)},\\
\min_{\theta \in [-\theta_1,\pi/2]}\cos\Big{(}\frac{\theta+\theta_1}{2}\Big{)} &= \cos\Big{(}\frac{\pi}{4} + \frac{\theta_1}{2}\Big{)}.
\end{align*}
Moreover, in view of the following inequalities
\begin{align*}
0 < \frac{\pi}{4} - \frac{\theta_1}{2} < \frac{\pi}{4},
\\
\frac{\pi}{4} < \frac{\pi}{4} + \frac{\theta_1}{2} < \frac{\pi}{2},
\\
\frac{\pi}{4} + \frac{\theta_1}{2} > \frac{\pi}{4} - \frac{\theta_1}{2},
\end{align*}
we conclude
\begin{equation*}
\cos\Big{(} \frac{\pi}{4} + \frac{\theta_1}{2}\Big{)} < \cos\Big{(} \frac{\pi}{4} - \frac{\theta_1}{2}\Big{)}.
\end{equation*}
Finally, we obtain
\begin{equation*}
\min_{\theta \in [-\pi/2,\pi/2]}\cos\Big{(}\frac{\theta+\theta_1}{2}\Big{)} = \cos\Big{(}\frac{\pi}{4} + \frac{\theta_1}{2}\Big{)},
\end{equation*}
that is, the minimum is attained for $\bar{\theta} = \pi/2$.
Similarly, it follows that
\begin{equation*}
\min_{\theta \in [-\pi/2,\pi/2]}\cos\Big{(}\frac{\theta+\theta_2}{2}\Big{)} = \cos\Big{(}-\frac{\pi}{4} + \frac{\theta_2}{2}\Big{)} = \cos\Big{(}-\frac{\pi}{4} - \frac{\theta_1}{2}\Big{)} =
\cos\Big{(}\frac{\pi}{4} + \frac{\theta_1}{2}\Big{)}
\end{equation*}
and therefore we can take
\begin{equation*}
\alpha = \Re(\exp(i(\bar\theta+\theta_1)/2));\ \bar\theta=\pi/2.
\end{equation*}
\begin{lemma}\label{upper_bound_lambda_c}
Suppose $\mu < 2$. 
Then, we can find a  constant $C>0$ such that  $\sqrt{2}\sqrt{\epsilon_{+}^2+\epsilon_{-}^2}<~1$ for all $\lambda$ with $\Re(\lambda)\geq 0$ and $|\lambda|\geq C$. As a consequence,  the Evans function $E(\lambda)$ for \eqref{scalar_eq_c}  has no zeros in this region.
\end{lemma}
\begin{proof}
As in the viscosity dominant case, we compute directly the entries of the matrix $S^{-1}B(x,\lambda)S$ in \eqref{diagonal_perturbed_system} and we observe that $\tilde{b}_{j,k}$ are upper bounds for their absolute values. 
The terms $m_k(\lambda)$ in \eqref{bounds} are upper bounds  for the terms $b_{4,k}(\lambda)$ of the matrix  $B(x,\lambda)$ from \eqref{perturbed_system}:
   $|b_{4,k}(x,\lambda)|\leq m_k(\lambda)$, while the terms $|z_k|$ and $|z_j-z_k|$ come from $S$ and $S^{-1}$.  Moreover, as in the proof of Lemma \ref{upper_bound_lambda_visc}, the matrices $\tilde{B}_{\pm}(\lambda)$ from \eqref{residual_matrix} 
 have  monotonically decreasing in $|\lambda|$ entries, therefore, from the bound above, we can easily find $C>0$ such that    $\sqrt{2}\sqrt{\epsilon_{+}^2+\epsilon_{-}^2}<~1$ for all $\lambda$ with $\Re(\lambda)\geq 0$ and $|\lambda|\geq C$ and the proof is complete.
\end{proof}
\subsection{Numerical evidence of point spectrum stability}\label{numerics}
To conclude our analysis leading to point spectrum stability, we shall now exclude the presence of  eigenvalue in a bounded region $|\lambda|\leq C$ inside the unstable half-plane, where $C$ is given in terms of the quantitative bound about the modulus of possible eigenvalues obtained  above. To this end, 
under the assumption that the Evans function $E(\lambda)$ is analytic in the region surrounded by a closed contour $\Gamma$, and it does not vanish on the contour, we can use the winding number
\begin{equation}
\label{log_idnicator1}
\frac{1}{2 \pi i} \int_{\Gamma}\frac{E'(z)}{E(z)}dz
\end{equation}
to count the number of zeros inside the contour.  The remaining part of this paper is devoted to provide numerical evidence that the integral \eqref{log_idnicator1} is indeed zero in a sufficiently large contour $\Gamma$ lying in the unstable half-plane, according to the aforementioned quantitative bound.

Specifically, to compute the Evans function numerically, we use the compound matrix method; for instance, see \cite{Humpherys}.
This method  is used in order to get a stable numerical procedure, in spite of the fact that the system $Y'=M(y,\lambda)Y$ is numerically stiff.
Specifically, the compound matrix $B(y,\lambda)$ is given by:\
\begin{equation*}
B=\begin{bmatrix}
    m_{11}+m_{22} & m_{23} & m_{24} & -m_{13} & -m_{14} & 0 \\
    m_{32} & m_{11}+m_{33} & m_{34} & m_{12} & 0 & -m_{14} \\
    m_{42} & m_{43} & m_{11}+m_{44} & 0 & m_{12} & m_{13} \\
    -m_{31} & m_{21} & 0 & m_{22}+m_{33} & m_{34} & -m_{24} \\
    -m_{41} & 0 & m_{21} & m_{43} & m_{22}+m_{44} & m_{23} \\
     0 & -m_{41} & m_{31} & -m_{42} & m_{32} & m_{33}+m_{44}
  \end{bmatrix}.
\end{equation*}
We integrate the equation $\phi'=(B(y,\lambda)-\mu^-)\phi$ numerically  on a sufficiently large interval $[-L_1,0]$, where $\mu^-$ is the unstable eigenvalue of $B$ at $-\infty$ with maximal (positive) real part. Denote the profile $[R(y), U(y)]$ by $\zeta(y)$. Given a numerical approximation of $(\zeta_k)_{k=1}^N$ of $\zeta(y)$ at points $(y_k)_{k=1}^N$ with $-L_1 = y_1 < y_2 < ... < y_N = L_1$, let $\tilde{\zeta}(y)$ be the piecewise linear interpolant of $(y_1,\zeta_1),...,(y_N,\zeta_N)$. We obtain the matrix $B(y,\lambda)$ using $\tilde{\zeta}(y)$.
Similarly we integrate the equation $\phi'=(B(y,\lambda)-\mu^+)\phi$ on $[0,L_1]$ backwards, where this time  $\mu^+$ is the stable eigenvalue of $B$ at $+\infty$ with minimal (negative) real part. Then, the coefficients $\mu^{\pm}$ compensate for the growth/decay at infinity. Finally, the Evans function can be constructed by means of linear combination of the components of the two solutions $\phi^\pm = (\phi^\pm_1, \dots, \phi^\pm_6)$ as follows:
\begin{equation*}
E(\lambda)=\phi^-_1\phi^+_6-\phi^-_2\phi^+_5+\phi^-_3\phi^+_4+\phi^-_4\phi^+_3-\phi^-_5\phi^+_2+\phi^-_6\phi^+_1\Big |_{y=0}.
\end{equation*}

For our calculations we use $L_1 = 40$ and we confine ourselves to 
 %
%
%
the following set of parameters, included in the  dispersion dominant case, as defined above:
\begin{equation*}
P^+ = 0.6,\mbox{ }P^- = 0.8,\mbox{ }s=1,\mbox{ } \gamma = 3/2,\mbox{ }\mu = 1,\mbox{ }k = \sqrt{2},
\end{equation*}
where $P^{\pm} = \sqrt{R^{\pm}}$.
The values for the velocity defining an admissible Lax 2--shock are $U^+ = -0.32$ and $U^- = 0.25$. The sufficient conditions for existence of profile of \cite[Lemma 1]{Zhelyazov1}, case (i) are verified. Moreover, since 
\begin{equation*}
|U^+| = 0.32 < 0.95 = c_s(R^+),
\end{equation*}
conditions (i) of \cite[Corollary 2]{Zhelyazov1} hold as well. Also the condition of \cite[Lemma 1]{Zhelyazov1}, case (i) is satisfied, because
\begin{equation*}
\frac{s \mu}{k} = 0.71< 1.3 = \sqrt{-f'(P^+)},
\end{equation*}
and the profile is non-monotone.

As pointed out already in Section \ref{subsec:intvar},  to avoid the smallness of $E(\lambda)$ near zero, we use integrated variables, namely, we solve the ODEs
$\phi'=(\hat B(y,\lambda)-\mu^\pm)\phi$, where the compound matrix $\hat B$ is constructed from $\hat M$ defined in \eqref{mat_M1_c}. 
To numerically check the Evans function is indeed well defined and different from zero at $\lambda =0$, we evaluate it on a small semi-circular contour without a vertical segment in the unstable half-plane with radius $10^{-6}$ and with center at $\lambda=0$, showing that it is almost constant (and non zero); see Figure \ref{figure_evans_function_small_contour}.
\begin{figure}[h]
\includegraphics[scale=0.6]{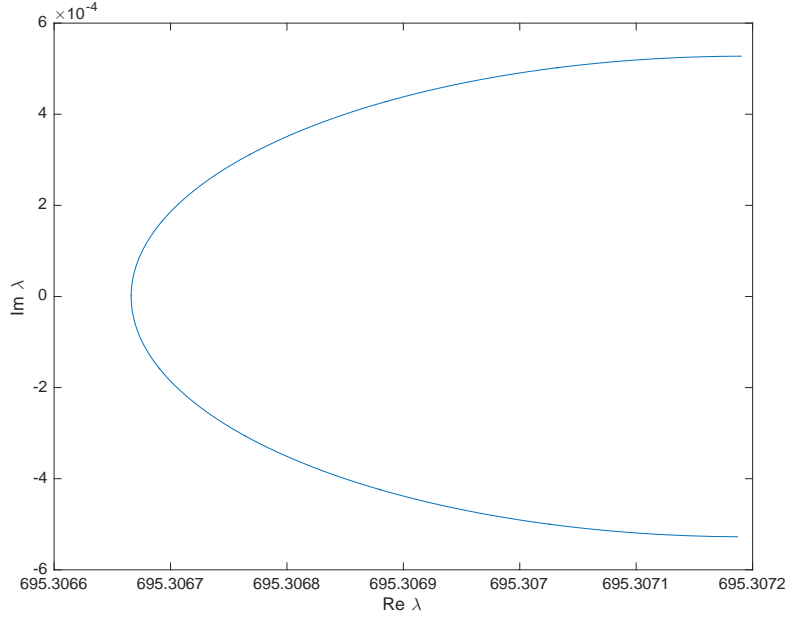}
\caption{The image of a small semi-circular contour without a vertical segment, with radius $10^{-6}$, and center at the origin through $E(\lambda)$.}
\label{figure_evans_function_small_contour}
\end{figure}

 To compute the initial conditions we integrate the reduced Kato ODE
\begin{equation*}
\frac{dr_\pm}{d\lambda}=\frac{d\mathcal{P}_\pm}{d\lambda}r_\pm,
\end{equation*}
where $\mathcal{P}_\pm$ stands for the spectral projection of $\hat{B}^{\pm} = \lim_{y \rightarrow \pm \infty}\hat{B}(y,\lambda)$ associated to $\mu^{\pm}$. For this, we use  the algorithm from \cite{Zumbrun}, that is $|r_\pm^1|=1$ eigenvector as before (referring to maximal/minimal decay/growth rate of $B^\pm$) and for $k>0$
\begin{equation*}
r_\pm^{k+1}=\mathcal{P}_\pm^{k+1}r_\pm^k.
\end{equation*}

It is worth observing that, since the Evans functions corresponding to the profiles $\zeta(y)$ and $\zeta(y-y_0)$ have the same zero set and the estimate for the constant $C$ that bounds the modulus of eigenvalues with nonegative real part provided by Lemma \ref{upper_bound_lambda_c} depends on $y_0$, it may be possible to obtain a smaller bound for $C$ by shifting the profile.
Therefore, let $y_0 \in (-L_1,L_1)$ so that the translated interpolant $\tilde{\zeta}(y-y_0)$ is defined on $[-L_1 + y_0,L_1 + y_0]$. We discretize the two domains $[-L_1 + y_0, 0]$ and $[0,L_1 + y_0]$ using uniform grids $(\tilde{y}_k)_{k=1}^{N_-}$ and $(\bar{y}_k)_{k=1}^{N_+}$, with
\begin{align*}
\tilde{y}_k &= -L_1 + y_0 + (k-1) \Delta y^-,\mbox{ }k=1,...,N_-,\\
\bar{y}_k &= (k-1) \Delta y^+,\mbox{ }k=1,...,N_+,
\end{align*}
where $\Delta y^{\pm}$ are the grid sizes. Then, we compute the suprema in \eqref{bounds} on the grids and construct the matrices $\tilde{B}_{\pm}$. Finally, we evaluate $\delta_{\pm}$ and $\epsilon_{\pm}$. It is sufficient to choose $\lambda \in \mathbb{R^+}$ sufficiently large, so that the condition of Lemma \ref{upper_bound_lambda_c} is satisfied.
We choose $y_0 = 10$ and, being the interpolant $\tilde{\zeta}(y)$  defined on $[-40,40]$, we have that  $\tilde{\zeta}(y-y_0)$ is defined on $[-30, 50]$. Moreover, we set $\Delta y^\pm= 0.1$. With these choices, using Lemma \ref{upper_bound_lambda_c}, we obtain numerically that  there are no eigenvalues for  $|\lambda|\geq 1.5 \cdot 10^4$. In the sequel we complement this information with a numerical evidence of absence of eigenvalues inside that circle.

We have to initialize the computation on the real axis and, for stability reasons, for a not very large value of $\lambda$. For these reasons, we cover the region  of the unstable half-plane inside $|\lambda| \leq 1.5 \cdot 10^4$ with the union of the areas surrounded by the following two contours:
\begin{enumerate}
\item One semi-circular contour with radius 10, center at $\lambda = 0$ and vertical segment on the imaginary axis. Here we do not evaluate the Evans function at $0$, but evaluate it up to $\pm i 10^{-6}$.
\item One contour which surrounds a semi-annular region in the right half-plane with two semi-circles with radii $5$ and $1.5 \cdot 10^4$, center at $\lambda = 0$ and vertical segment on the imaginary axis. 
\end{enumerate}
Along the first contour we integrate the Kato ODE using $4 \cdot 10^4$ points, while along the second one we use $10^6$ points with higher density near the origin.
 Then, using these initial conditions, we compute $E(\lambda)$ with the stiff solver ode15s in matlab, with relative tolerance $10^{-6}$ and, as said before, we set $L_1=40$. 
Finally, we apply the symmetry of $E(\bar{\lambda})=\overline{E(\lambda)}$.
The Evans function $E(\lambda)$ is plotted in Figures \ref{figure_evans_fun_2} and \ref{figure_evans_fun_2b} and its winding number is (approximately) 0, giving a numerical evidence of point spectrum stability.

Moreover, we present a computation of the Evans function along a contour, surrounding a semi-annular region with radii $10^{-6}$ and $1.5 \cdot 10^4$, center at $\lambda = 0$ and vertical segment on the imaginary axis.  Along the contour we integrate the Kato ODE with $10^6$ points; see  Figure \ref{evans_function_annular_small_circle}. Again, the winding number of the Evans function is (numerically) 0.

Finally, we note that  if $\lambda \in \mathbb{R}$, then $E(\lambda) \in \mathbb{R}$. Our numerics agrees with this simple observation, because we get  $\Im{E(\lambda)} \approx 0$ for $\lambda = 1.5\cdot10^4$. Moreover, to corroborate this accuracy, we also use the Cauchy integral formula
\begin{equation*}
E(a) = \frac{1}{2 \pi i}\int_{\Gamma}\frac{E(z)}{z-a}dz
\end{equation*}
for $a$ inside the contour $\Gamma$ surrounding the semi-annular region with inner radius $10^{-6}$, and,  for $a=1.5\cdot10^4-20$, we get a relative error less than $5\cdot10^{-4}$.

\begin{figure}[h]
\begin{center}
\centering
\includegraphics[scale=0.6]{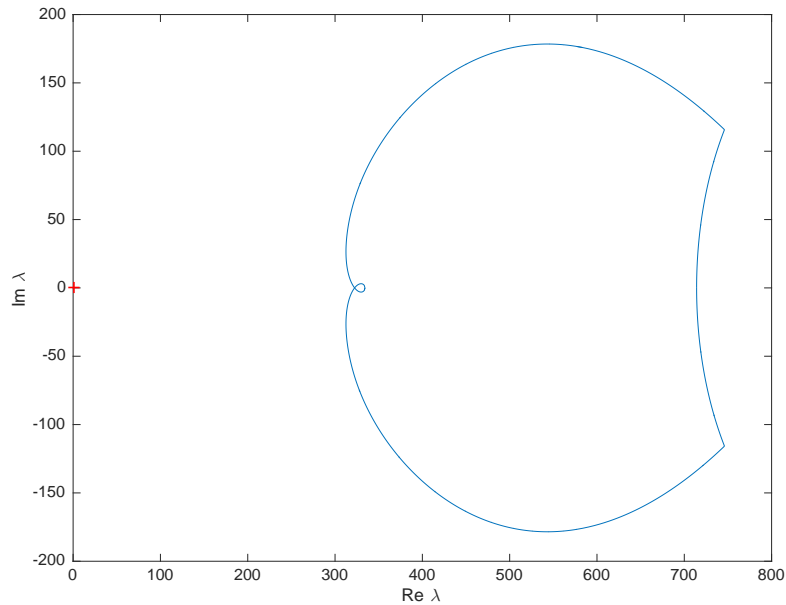}
\caption{The image of a semi-circular contour with radius $10$ through the Evans function $E(\lambda)$. The origin is marked in red.}
\label{figure_evans_fun_2}
\end{center}
\end{figure}
\begin{figure}[h]
\begin{center}
\centering
\includegraphics[scale=0.6]{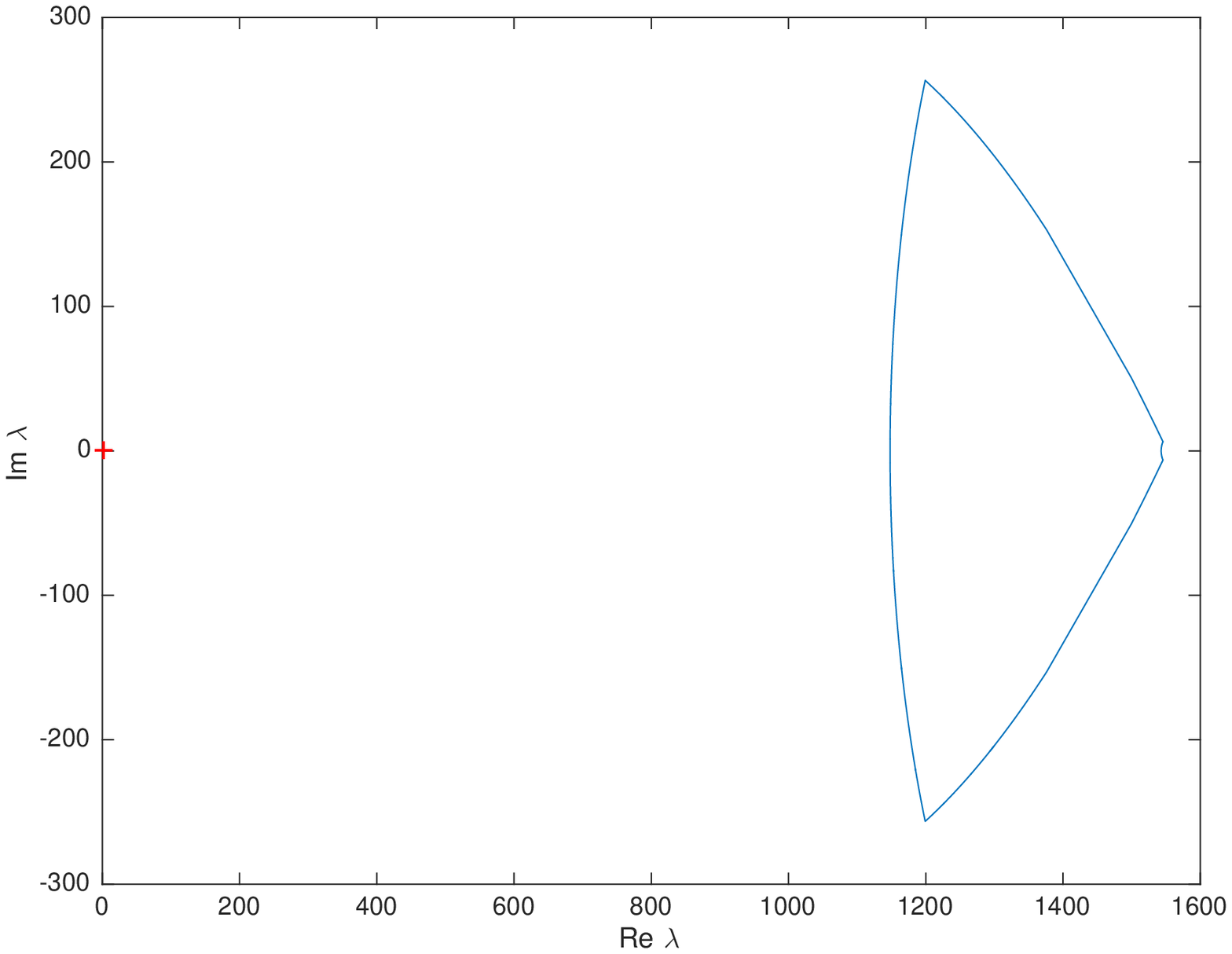}
\caption{The image of a contour, surrounding a semi-annular region with radii $5$ and $1.5\cdot10^4$ through $E(\lambda)$. The origin is marked in red.}
\label{figure_evans_fun_2b}
\end{center}
\end{figure}
\begin{figure}[h]
\centering
\includegraphics[scale=0.6]{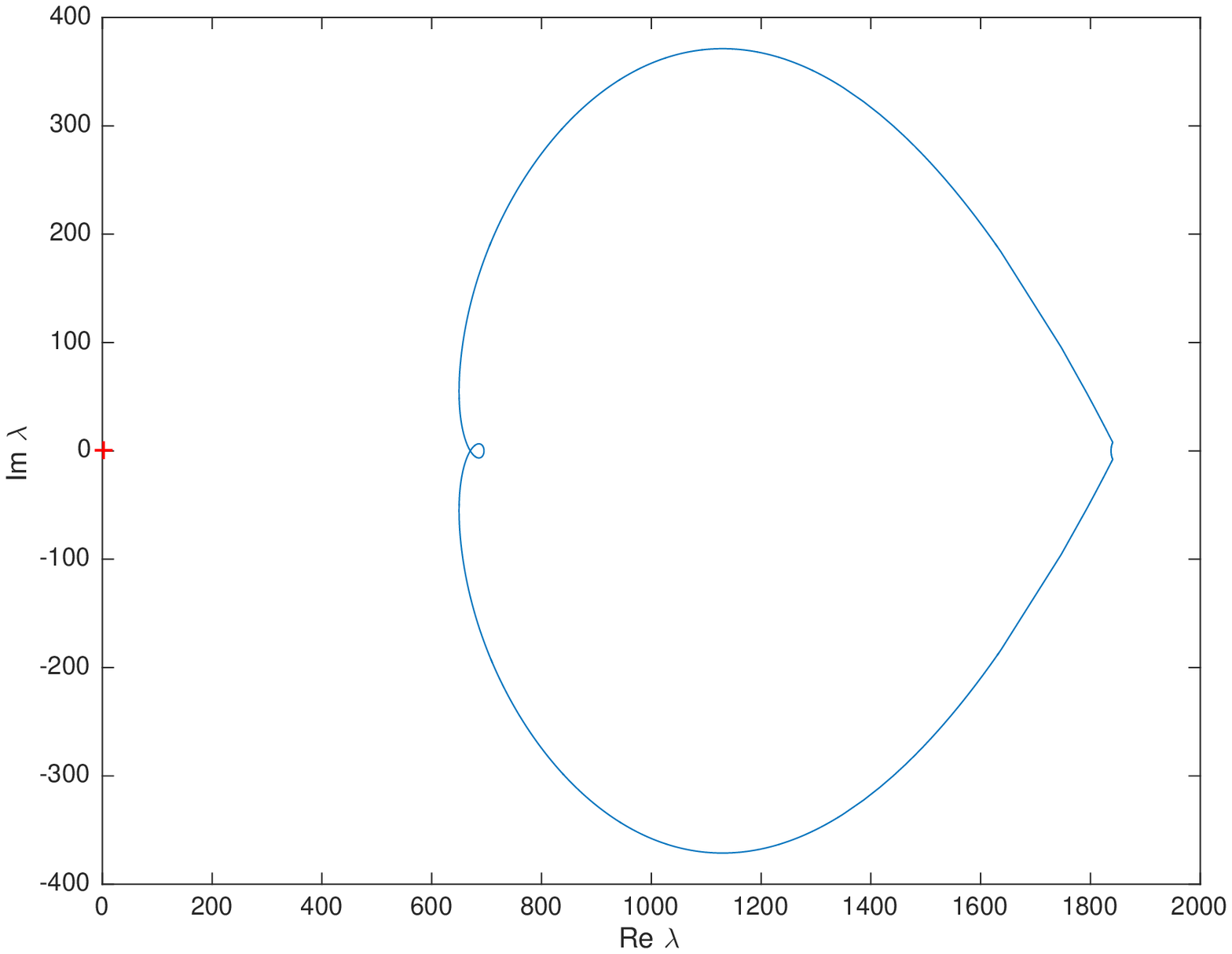}
\caption{The image of a contour, surrounding a semi-annular region with radii $10^{-6}$ and $1.5\cdot10^4$ through $E(\lambda)$. The origin is marked in red.}
\label{evans_function_annular_small_circle}
\end{figure}


\end{document}